\documentclass[reqno]{amsart}

\usepackage[all]{xy}
\usepackage{graphicx}

\usepackage{amssymb}
\usepackage{amsmath}
\usepackage{mathrsfs}
\usepackage{epsfig}
\usepackage{amscd}
\usepackage{graphicx, color}

\def\E{\ifmmode{\mathbb E}\else{$\mathbb E$}\fi} %natural numbers
\def\N{\ifmmode{\mathbb N}\else{$\mathbb N$}\fi} %natural numbers%
\def\R{\ifmmode{\mathbb R}\else{$\mathbb R$}\fi} %real numbers
\def\Q{\ifmmode{\mathbb Q}\else{$\mathbb Q$}\fi} %rational numbers
\def\C{\ifmmode{\mathbb C}\else{$\mathbb C$}\fi} %complex numbers
\def\H{\ifmmode{\mathbb H}\else{$\mathbb H$}\fi} %complex numbers
\def\Z{\ifmmode{\mathbb Z}\else{$\mathbb Z$}\fi} %integers
\def\P{\ifmmode{\mathbb P}\else{$\mathbb P$}\fi} %real numbers
\def\T{\ifmmode{\mathbb T}\else{$\mathbb T$}\fi} %real numbers
\def\SS{\ifmmode{\mathbb S}\else{$\mathbb S$}\fi} %real numbers
\def\DD{\ifmmode{\mathbb D}\else{$\mathbb D$}\fi} %real numbers

\renewcommand{\b}{\beta}

\renewcommand{\o}{\omega}

\newcommand{\del}{\partial}

\newcommand{\ben}{\begin{enumerate}}
\newcommand{\een}{\end{enumerate}}
\newcommand{\be}{\begin{equation}}
\newcommand{\ee}{\end{equation}}
\newcommand{\bea}{\begin{eqnarray}}
\newcommand{\eea}{\end{eqnarray}}
\newcommand{\beastar}{\begin{eqnarray*}}
\newcommand{\eeastar}{\end{eqnarray*}}
\newcommand{\bc}{\begin{center}}
\newcommand{\ec}{\end{center}}

\theoremstyle{theorem}
\newtheorem{thm}{Theorem}[section]
\newtheorem{cor}[thm]{Corollary}
\newtheorem{lem}[thm]{Lemma}
\newtheorem{prop}[thm]{Proposition}

\theoremstyle{definition}
\newtheorem{defn}[thm]{Definition}
\newtheorem{rem}[thm]{Remark}

\newtheorem*{thm*}{Theorem}

\numberwithin{equation}{section}

\def\R{{\mathbb R}}

\def\Crit{{\hbox{Crit}}}

\def\E{{\mathbb E}}
\def\Z{{\mathbb Z}}
\def\C{{\mathbb C}}
\def\R{{\mathbb R}}
\def\P{{\mathbb P}}

\def\N{{\mathbb N}}

\def\11{{\mathbb I}}

\def\delbar{{\overline \partial}}

\def\C{\mathbb{C}}
\def\Z{\mathbb{Z}}

\def\T{\mathbb{T}}

\def\Q{\mathbb{Q}}

\def\E{\ifmmode{\mathbb E}\else{$\mathbb E$}\fi} %natural numbers
\def\N{\ifmmode{\mathbb N}\else{$\mathbb N$}\fi} %natural numbers
\def\R{\ifmmode{\mathbb R}\else{$\mathbb R$}\fi} %real numbers
\def\Q{\ifmmode{\mathbb Q}\else{$\mathbb Q$}\fi} %rational numbers
\def\C{\ifmmode{\mathbb C}\else{$\mathbb C$}\fi} %complex numbers
\def\H{\ifmmode{\mathbb H}\else{$\mathbb H$}\fi} %complex numbers
\def\Z{\ifmmode{\mathbb Z}\else{$\mathbb Z$}\fi} %integers
\def\P{\ifmmode{\mathbb P}\else{$\mathbb P$}\fi} %real numbers
\def\SS{\ifmmode{\mathbb S}\else{$\mathbb S$}\fi} %real numbers
\def\DD{\ifmmode{\mathbb D}\else{$\mathbb D$}\fi} %real numbers

\def\R{{\mathbb R}}

\def\Crit{{\hbox{Crit}}}
\def\E{{\mathbb E}}
\def\Z{{\mathbb Z}}
\def\C{{\mathbb C}}
\def\R{{\mathbb R}}

\def\N{{\mathbb N}}
\def\MM{{\mathcal M}}

\def\JJ{{\mathcal J}}

\def\delbar{{\overline \partial}}

\def\b{\beta}

\def\o{\omega}  

  \def\S{\Sigma}

\def\CB{{\mathcal B}}

\def\CF{{\mathcal F}}

\def\CH{{\mathcal H}}

\def\CJ{{\mathcal J}}

\def\CM{{\mathcal M}}

\def\darr#1{\raise1.5ex\hbox{$\leftrightarrow$}
\mkern-16.5mu #1}

\def\roughly#1{\raise.3ex\hbox{$#1$\kern-.75em
\lower1ex\hbox{$\sim$}}}

\def\opname#1{\mathop{\kern0pt{\rm #1}}\nolimits}

\def\dim{\opname{dim}}

\def\supp{\operatorname{supp}}

\def\Aut{\operatorname{Aut}}
\def\Conf{\operatorname{Conf}}
\def\ram{\operatorname{ram}}
\def\Crit{\operatorname{Crit}}

\begin{document}
\quad \vskip1.375truein

\def\mq{\mathfrak{q}}
\def\mp{\mathfrak{p}}
\def\mH{\mathfrak{H}}
\def\mh{\mathfrak{h}}
\def\ma{\mathfrak{a}}
\def\ms{\mathfrak{s}}
\def\mm{\mathfrak{m}}
\def\mn{\mathfrak{n}}
\def\mz{\mathfrak{z}}
\def\mw{\mathfrak{w}}
\def\Hoch{{\tt Hoch}}
\def\mt{\mathfrak{t}}
\def\ml{\mathfrak{l}}
\def\mT{\mathfrak{T}}
\def\mL{\mathfrak{L}}
\def\mg{\mathfrak{g}}
\def\md{\mathfrak{d}}
\def\mr{\mathfrak{r}}

\title[Higher jet evaluation transversality]
{Higher jet evaluation transversality of $J$-holomorphic curves}

\author{Yong-Geun Oh}
\address{Department of Mathematics, University of Wisconsin, Madison, WI, 53706}
\email{oh@math.wisc.edu}
\thanks{Partially supported by the NSF grant \#DMS 0503954}

\begin{abstract}
In this paper, we establish general stratawise higher jet evaluation
transversality of $J$-holomorphic curves for a generic choice of
almost complex structures $J$ (tame to a given symplectic manifold
$(M,\omega)$).

Using this transversality result, we prove that there exists a subset
$\CJ_\omega^{ram} \subset \CJ_\omega$ of second category such that
for every $J \in \CJ_\omega^{ram}$, the dimension of the
moduli space of (somewhere injective) $J$-holomorphic curves with
a given ramification profile goes down by $2n$ or $2(n-1)$ depending on
whether the ramification degree goes up by one or a new ramification
point is created.

We also derive that for each $J \in \CJ_\omega^{ram}$ there are only a finite number of
ramification profiles of $J$-holomorphic curves in a given
homology class $\beta \in H_2(M;\Z)$ and provide an
explicit upper bound on the number of ramification profiles in terms of
$c_1(\beta)$ and the genus $g$ of the domain surface.
\end{abstract}

\date{April 16, 2009}


\maketitle

\tableofcontents

\section{Introduction}
\label{sec:intro}

Let $(M,\omega)$ be a symplectic manifold of dimension $2n$. Denote by $J$
an almost complex structure tame to $\omega$ and by $\CJ_\omega$ the set of
tame almost complex structures.

Let $\Sigma$ be an oriented compact surface without boundary of genus
$g$ and $(j,u)$ be a pair of a complex structure $j$ on $\Sigma$ and
a map $u: \Sigma \to M$. We say that $(j,u)$ is a $J$-holomorphic if
it satisfies $J \circ du = du \circ j$. We denote the standard moduli
spaces of $J$-holomorphic maps $(j,u)$ from $\Sigma$ to $M$ in class
$[u] = \beta$ by $\widetilde \CM_g(M,J;\beta)$ and consider
its quotient $\CM_g(M,J;\beta) = \widetilde
\CM_g(M,J;\beta)/\operatorname{Aut}(\Sigma)$.

The main purpose of the present paper is to establish higher jet evaluation
transversality whose precise formulation we refer to section \ref{sec:stratawise}.

From this higher jet evaluation transversality, we derive
stratawise transversality of ramification divisors whose statement is
now in order.

\begin{defn} The \emph{ramification degree} of the map $u$ at a point $z \in \Sigma$ is
defined to be the unique integer $k \in \N$ such that
$$
j^ku(z) = 0, \quad \mbox{but }\,  j^{k+1} u (z) \neq 0
$$
where $j^ku(z)$ is the $k$-jet of the map $u$ at $z$. If there is no such $k$, we say
$u$ has an infinite ramification degree. We say that any immersed
point has ramification degree 0.
\end{defn}

Basic results from \cite{mcduff}, \cite{sikorav} on the structure of singularities
of $J$-holomorphic map $(j,u)$ state that there are only finitely many
critical points and that each critical point has a finite ramification
degree. This motivates us to consider the set of pairs
$$
(k; \vec n),  \, \; \, k \in \N \, \mbox{and }\, \vec n \in \N^k.
$$
For each given $k$ distinct points
$z = \{z_1, \cdots, z_k\}$, we consider the decoration of positive
integers $n_i$ assigned at $z_i$'s. We denote $\vec n = \{n_1, \cdots, n_k\}$
and $K = \{1, \cdots, k\}$. For given $k \leq k'$ and $K' = \{1, \cdots, k'\}$
we decompose
$$
K' = K \cup (K' \setminus K).
$$
\begin{defn} We say $(k';\vec n') < (k; \vec n)$ if
$$
k \leq k' \, \mbox{ and } \quad n_i \leq n'_i\, \mbox{ for all } \,
i \in K \subset K'.
$$
\end{defn}

Now for each given non-constant $J$-holomorphic map $(j,u)$, we associate to it
the \emph{ramification profile} given by the vector
\be\label{eq:profile}
\vec n \in \coprod_{m\in \N} \N^m.
\ee
When $\vec n \in \N^k$, we also denote it by $(k;\vec n)$.
Denote by $\ram(u)$ the ramification profile of $u$ and
by $\deg_{z_i}(u)$ the ramification degree of $u$ at $z_i$.

We define the corresponding moduli space of $J$-holomorphic maps with
prescribed ramifications at $k$ marked points
$$
\widetilde\CM_{g,k}(M,J;\beta,\vec n) = \{ ((j,u),z) \mid \delbar_{(j,J)}u = 0, \,
\deg_{z_i}(u) \geq n_i \}
$$
and
$$
\CM_{g,k}(M,J;\beta,\vec n) = \widetilde\CM_{g,k}(M,J;\beta,\vec n)/\mbox{Aut}(\Sigma).
$$
We emphasize, though, that \emph{an element $u$ from this moduli space could have other
ramification points unmarked}.
\medskip

\noindent{\bf Notation :} {\it Throughout the paper, we will abuse our
notation and always denote by $\MM_{g,k}(M,J;\beta)$ the open subset
consisting of somewhere injective $J$-holomorphic curves $(j,u)$
in the standard smooth moduli space which is usually denoted by
the notation $\MM_{g,k}(M,J;\beta)$ itself. Similar remarks will apply to all other moduli spaces.}
\medskip

The following is one of the main theorems we prove in the present paper.

\begin{thm} There exists a subset $\JJ_\omega^{ram} \subset \JJ_\omega$
such that for any $J \in \JJ_\omega^{ram}$ the moduli space
$\widetilde \CM_{g,k}(M,J;\beta;\vec n)$ is a smooth manifold
of dimension
$$
\dim \widetilde \CM_{g,k}(M,J;\beta) - \sum_{i=1}^k 2n n_i
$$
for all $\beta \in H_2(M)$ and $g \in \Z_{\geq 0}$.
\end{thm}

Recall the standard dimension formula for

\be\label{eq:index} \dim \widetilde \CM_{g,k}(M,J;\beta) =
\begin{cases} 2(c_1(M,\omega)(\beta) + (n-3)(1-g)) + 2k \quad & \mbox{for }\, g\geq 2 \\
2(c_1(M,\omega)(\beta) + 1) + 2k \quad & \mbox{for }\, g = 1 \\
2(c_1(M,\omega)(\beta) + n) + 2k \quad & \mbox{for }\, g = 0
\end{cases}
\ee for the maps $u$ with $[u] = \beta\in H_2(\Z)$,

Our proof of this theorem relies on a new Fredholm set-up we
establish in this paper using the notion of \emph{holomorphic jet
bundles}. Using this Fredholm work and some judicious usage of a
structure theorem of distributions with point support (see \cite{gelfand}
for example), we prove a higher jet evaluation transversality which uses an extension of
the scheme of the 1-jet transversality proof employed by Zhu and the present
author in \cite{oh-zhu}. An important point used in our proof is the fact
that the \emph{holomorphic jet bundles} are canonically associated to
the pair of a Riemann surface $(\Sigma,j)$ and an almost complex manifold $(M,J)$
in the `off-shell' level, i.e. on the space of smooth maps, not just on the
moduli space of $J$-holomorphic maps.

A priori, $\widetilde \CM_{g,k}(M,J;\beta;\vec n)$'s are abstract manifolds
residing independently from one another. The following theorem
relates them when the corresponding ramification orders are
right next to each other.

We have two kinds of immediate predecessors $(k';\vec n')$ to $(k;\vec n)$ :
\begin{enumerate}
\item[(a)] $(k'; \vec n') = (k; \vec n + \vec e_\ell)$ for some $1 \leq \ell \leq k$
where we denote by $\vec n + \vec e_\ell$ the decoration
$$
(n_1, \cdots, n_\ell + 1, \cdots, n_k).
$$
\item[(b)] $(k'; \vec n') = (k+1, \vec n \cup \{n_{k+1}\})$
with $n_{k+1} = 1$.
\end{enumerate}

\begin{thm} For $J \in \JJ_\omega^{ram}$ and $\beta \in H_2(M)$ and $g \in \Z_{\geq 0}$,
the following holds :
\begin{enumerate}
\item For the type $(a)$ of the immediate predecessor of $(k';\vec n')
= (k,\vec n + \vec e_\ell)$ for some $\ell = 1, \cdots, k$, $
\MM_{g,k}(J;\beta, \vec n + \vec e_\ell)$ is a smooth submanifold of
$\MM_{g,k}(J;\beta, \vec n)$ with its dimension $2n$ smaller,
\item For the type $(b)$, the image of the forgetful map
$\MM_{g,k+1}(J;\beta,\vec n + \vec e_{k+1}) \to \MM_{g,k}(J;\beta,\vec n)$
induces an embedding of codimension $2(n-1)$.
\end{enumerate}
\end{thm}

It has been a folklore  that ``for a generic choice of $J$, the
dimension of the moduli space of a given ramification profile goes
down when either the ramification order goes up or a new
ramification point is created''. However it has not been clear what
the precise statement of this folklore would really be. The above
theorem provides a precise form of this folklore. The main stumbling
block to make this folklore into a precise theorem has been what
kind of moduli spaces one should look at to obtain the kind of anticipated
dimension cutting-down statement hold. For example, it will become
clear in the course of our proof that the folklore cannot be
formulated in terms of moduli space of unmarked holomorphic maps. It
took the author some time to find out which moduli space is the
correct one with respect to which the necessary Fredholm framework
can be carried out. Only after the work \cite{oh-zhu} which concerns
the 1-jet evaluation transversality, the answer became clear to the
author. This has led the author to the Fredholm setting used in the
present paper.

Next we study the cardinality of ramification profiles of (unmarked) $J$-holomorphic curves
for a given genus $g$ and homology class $\beta \in H_2(M,\Z)$.
We have the obvious decomposition of the moduli space of unmarked $J$-holomorphic curves
\be\label{eq:union}
\widetilde \CM_g(M,J;\beta) = \bigcup_{k \in \N} \bigcup_{\vec n \in \N^k}
\widetilde \CM_g^{(k;\vec n)}(M,J;\beta)
\ee
where $\widetilde \CM_g^{(k;\vec n)}(M,J;\beta)$ is the subset of $\widetilde \CM_g(M,J;\beta)$ given by
$$
\widetilde \CM_g^{(k;\vec n)}(M,J;\beta) = \{ (j,u) \in \widetilde \CM_g(M,J;\beta) \mid
\ram(u) = (k;\vec n) \}.
$$
The relation between $\widetilde \CM_g^{(k;\vec n)}(M,J;\beta)$ and
$\widetilde \CM_{g,k}(M,J;\beta;\vec n)$ is the following :
Consider the forgetful map
\be\label{eq:forgetful}
\frak{forget}_{(k;\vec n)} :\widetilde \CM_{g,k}(M,J;\beta;\vec n)
\to \widetilde \CM_g(M,J;\beta).
\ee
Then we have
\beastar
\widetilde \CM_g^{(k;\vec n)}(M,J;\beta) & = &
\frak{forget}_{(k;\vec n)}
\left(\widetilde \CM_{g,k}(M,J;\beta;\vec n)\right) \\
&{}& \quad \Big \backslash \bigcup_{(k';\vec n') < (k;\vec n)}
\frak{forget}_{(k';\vec n')}\left(\widetilde
\CM_{g,k'}(M,J;\beta;\vec n')\right). \eeastar We note that \emph{an
element from $\widetilde \CM_g^{(k;\vec n)}(M,J;\beta)$ will have
their ramification points and ramification orders exactly the same
as prescribed.} A priori, the union \eqref{eq:union} could be an infinite union.

The following theorem says that
this will be a finite union for a generic choice of $J$.

\begin{thm}\label{thm:finite}
Let $\beta \in H_2(M,\Z)$ and $g$ be given.
Then for any $J \in \CJ_\omega^{ram}$,  the number of
types of ramification profiles of $\CM_g(M,J;\beta)$ is not bigger than
$$
P(c_1(\b)+(3-n)(g-1)),
$$
that is, the number of partitions of the integer
$c_1(\b)+(3-n)(g-1)$, when $c_1(\b)+(3-n)(g-1) \geq 0$.
\end{thm}

We note that if $c_1(\b)+(3-n)(g-1) - n < 0$, then the corresponding moduli
space will be empty. And we also emphasize that each stratum
of the union \eqref{eq:union}
could have singularities as a subset of $\widetilde \CM(M,J;\beta)$.

The statements in the above theorems are somewhat
reminiscent of the N\"otherian
property of holomorphic maps in projective algebraic varieties.
We find it curious that this kind of finite statements hold
in two opposite ends of the category and wonder if there is any universal
phenomenon in that direction.

The Fredholm framework and the scheme of the relevant
evaluation transversality proof that we employ in the present paper
are the higher jet analogs to the ones used for the 1 jet evaluation transversality
studied in \cite{oh-zhu}. A similar higher jet analysis is also
carried out in \cite{oh:seidel} in relation to the compactification of the moduli
space of smooth holomorphic sections of the (singular)
Lefschetz Hamiltonian fibrations.

We thank Zhu for having many enlightening discussions during the collaboration of the
work \cite{oh-zhu} and other projects.

\section{Jet evaluation map and holomorphic jets}
\label{sec:jet}

In this section, we study a smooth map $u: \Sigma \to M$ whose
first $k$ derivatives vanish $z \in \Sigma$, i.e.,
$$
j^ku(z) = 0
$$
where $j^ku(z)$ is the $k$-jet of the map $u$. For such a map, we
would like to say that the $(k+1)$-th derivative of $(j,J)$-holomorphic map
$u$ at $z$
induces a canonical linear map from $T_z \Sigma \to H^{(k+1,0)}_{z,u(z)} \cong
T_{u(z)}M$ where $H^{(k+1,0)}_{z,u(z)}$ is the set of `holomorphic
part' of the $(k+1)$-jet space.
We will make this statement precise in the rest of the section.

For this purpose, we recall the definition of $k$-jet bundle and the
$k$-jet $j^ku(z)$ at $z \in \Sigma$. (See \cite{hirsch} for a nice
exposition on the jet bundle.) The $k$-jet bundle $J^k(\Sigma,M) \to
\Sigma \times M$ is the vector bundle whose fiber at $(z,x)$ is
given by
$$
J^k_{(z,x)}(\Sigma,M) = P^k(T_z\Sigma, T_xM)
$$
where $P^k(T_z\Sigma, T_xM)$ is the set of polynomial maps from $T_z\Sigma$
to $T_xM$ of degree $\leq k$, or
$$
P^k(T_z\Sigma, T_xM) =\prod_{\ell = 0}^k Sym^\ell(T_z\Sigma, T_xM)
$$
where $Sym^\ell(T_z\Sigma, T_xM)$ is the vector space of symmetric
$\ell$-linear maps from $T_z\Sigma \to T_xM$.

We have a natural sequence of bundles over $\Sigma \times M$
$$
J^0(\Sigma,M) \leftarrow J^1(\Sigma,M) \leftarrow \cdots\leftarrow  J^k(\Sigma,M) \leftarrow
$$
and
$$
J^0_x(\Sigma,M) \leftarrow J^1_x(\Sigma,M) \leftarrow \cdots \leftarrow J^k_x(\Sigma,M) \leftarrow
$$
Here the map $\pi:J^{k+1}(\Sigma,M) \to J^k(\Sigma,M)$ is defined by the `truncation' of
polynomial map $P$ to the terms of order $\leq k$. We denote by $P^{\leq k}$
the truncation of $P$ thereto. Then we have
$$
\pi(z,x;P) = (z,x;P^{\leq k})
$$

Now we consider the space
$$
\CF_1(\Sigma,M) = \{ (u,z) \mid u: \Sigma \to M, \, z \in \Sigma\}.
$$
We have the natural $k$-jet evaluation map
$$
j^k : \CF_1(\Sigma,M) \to J^k(\Sigma,M) \,;\, j^k(u,z) = j^k_zu
$$
and for a fixed point $z \in \Sigma$
$$
j^k_z : \CF(\Sigma,M) \to J^k_{z,(\cdot)(z)}(\Sigma, M)\, ;\, j^k_z(u) = j^k_zu.
$$

We have a short exact sequence
$$
0 \to \ker \pi \to J^{k+1}(\Sigma,M) \stackrel{\pi}{\to}
J^k(\Sigma,M) \to 0
$$
of vector bundles over $\Sigma \times M$. By construction, we have
\be\label{pijk+1=jk}
\pi \circ j^{k+1} = j^k.
\ee

Now we equip $\Sigma$ and $M$ with almost complex structures $j$ and
$J$ respectively. The almost complex structures $j$ on $\Sigma$ and
$J$ on $M$ naturally split off the direct summands of $Sym^k(\Sigma,M)$
such as
$$
Sym^{k}(\Sigma, M) = Sym^{(k,0)}_{(j_z,J_x)} (\Sigma,M) \oplus
Sym^{(0,k)}_{(j_z,J_x)}(\Sigma,M) \oplus ``\mbox{mixed parts}".
$$
\begin{defn} We call an element $L \in Sym^k(T_z\Sigma, T_xM)$
\emph{holomorphic} at $(z,x)$ (relative to $(j,J)$) if
$L$ lies in $H^{(k,0)}_{z,x}(\Sigma,M) := Sym^{(k,0)}_{(j_z,J_x)}(\Sigma,M)$.
\end{defn}
We note that the vector space $Sym^{(k;0)}_{(j_z,J_x)}(\Sigma,M)$ has the same dimension
as $T_xM$  for all $k$, which is nothing but $2n$.

We denote
$$
J^k_{hol,(z,x)}(\Sigma,M) = \bigoplus_{\ell = 0}^k
H^{(\ell,0)}_{(z,x)}(\Sigma,M)
$$
and form the union
$$
J^k_{hol}(\Sigma,M) = \bigcup_{(z,x) \in \Sigma \times M} J^k_{hol,(z,x)}(\Sigma,M)
$$
which we call the \emph{holomorphic jet bundle} of $\Sigma \times M$ relative to
$(j,J)$.

Then we have the natural projection
$$
\pi_{(j,J)}^{hol}: J^k(\Sigma,M) \to J^k_{hol}(\Sigma,M)
$$
and the inclusion
$$
i_{(j,J)}^{hol}: J^k_{hol} (\Sigma,M) \to J^{k}(\Sigma,M).
$$

Now consider a $C^{k+1}$-map $u: \Sigma \to M$. We define
the \emph{$k$-th holomorphic jet} of $u:\Sigma \to M$ by
$$
j^k_{hol} u(z) : =  \pi_{(j,J)}^{hol}(j^k u(z)) =
\bigoplus_{\ell=0}^k \pi_{j,J}^{hol}(d^\ell u(z))
$$
and
$$
\sigma^k(J,(j,u),z) : = \pi_{(j,J)}^{hol}(d^{k}u(z)).
$$

\begin{defn} Let $u: \Sigma \to M$ be a smooth map satisfying
$$
j^k u(z) = 0 \quad \mbox{ but } j^{k+1}u(z) \neq 0.
$$
We say that $u$ has \emph{ramification degree} $k$ and $j^{k+1}u(z)$
the \emph{principal jet} of $u$. The \emph{holomorphic
ramification degree} is defined similarly by using holomorphic jets
$j^k_{hol}u$ instead of $j^ku$.

For a map $u$ with ramification degree $k$, we call
$\sigma^{k+1}(J,(j,u),z)$ the \emph{principal holomorphic jet}
of the map $u$ at $z$.
\end{defn}

We next describe the holomorphic principal jet of a smooth map
$u:\Sigma \to M$ relative to $(j,J)$ in complex coordinates.
We refer to \cite{mcduff-92} for further details of some relevant exposition.

Let $z = s+ it$ be a complex coordinate of $(\Sigma,j)$ centered at
$z_0$, and choose real coordinates $(x_1,\cdots, x_n, y_1, \cdots, y_n)$ of $M$
centered at $p_0=u(z_0)$ for $j = 1, \cdots, n$ such that
\be\label{eq:complexatp0}
J\frac{\del}{\del x_j}\Big|_{p_0} = \frac{\del}{\del y_j}\Big|_{p_0}, \quad
J\frac{\del}{\del y_j}\Big|_{p_0} = - \frac{\del}{\del x_j}\Big|_{p_0}
\ee
We denote the associated complex coordinates by
$(w_1,\cdots, w_n)$ with $w_j = x_j + \sqrt{-1}y_j$. If $j^ku(z_0) = 0$, then we can
expand the map $u$ into the Taylor polynomial
$$
u(z) = \sum_{\ell = 0}^{k+1} \vec a_\ell z^\ell \overline z^{k-\ell} +
o(|z|^{k+1}).
$$
Combining \eqref{eq:complexatp0} and $j^ku(z) = 0$, we derive

\begin{lem}
Suppose $u$ is $(j,J)$-holomorphic, i.e., $\delbar_{(j,J)}u = 0$ and $j^ku(z) = 0$. Then
we have
$$
u(z) = \vec a_{k+1} z^{k+1} + o(|z|^{k+1})
$$
with $a_{k+1} \in T_{p_0}M \cong H_{z_0}^{(k;0)}$.
In particular, the principal jet of $(j,J)$-holomorphic map $u$ is holomorphic
at any point $z$, and the ramification degree of $u$ is the same as
the holomorphic ramification degree at any given point $z \in \Sigma$.
\end{lem}

It is easy to check that the principal term $\vec a_{k+1} z^{k+1}$,
regarded as an element in
$Sym^{k+1}_{(j_{z_0},J_{u(z_0)})}(T_{z_0}\Sigma,T_{u(z_0)}M)$, has
the from
$$
j^{k+1}u(z_0) = d^{k+1}u(z_0) = \vec a \cdot dz^{\otimes (k+1)}, \quad \vec a \in T_{u(z_0)}M
$$
for any map $(j,J)$-holomorphic map $u$ with $j^ku(z_0) = 0$.

Now, we immediately obtain the following characterization for a
$(j,J)$-holomorphic map $u$ with $j^ku (z_0) = 0$ to satisfy
$j^{k+1} u(z_0) = 0$.

\begin{lem}\label{lem:dk+1u=0} Let $u$ be a $(j,J)$-holomorphic map with
$j^ku(z_0) = 0$. Then $j^{k+1}u(z_0) = 0$ if and only if
$
\sigma^{k+1}(J,(j,u), z_0) = 0.
$
\end{lem}

We note that when $k =1$, $\sigma(J,(j,u),z) = \del_{(j,J)}u(z)$,
the holomorphic part of the derivative $du(z_0)$. This lemma is the
higher jet analog to Lemma 2.2 \cite{oh-zhu} which will be the basis
of our Fredholm setting for the evaluation transversality in higher
jets.

\section{Fredholm framework}
\label{sec:Fredholm}

Let $\Sigma$ be a compact orientable surface without boundary. We
consider a triple $(J,(j,u),z)$ of compatible $J$ on $M$, $j$
complex structure on $\Sigma$, $u : \Sigma \to M$ a smooth map and a
point $z \in \Sigma$.

Denote \beastar \CF(\Sigma,M;\beta) & = & \{(j,u) \mid j \in
\CM(\Sigma),\, u: \Sigma \to M, [u]=\beta\} \\
\CF_1(\Sigma,M;\beta) & = & \{ ((j,u),z) \mid (j,u) \in
\CF(\Sigma,M;\beta), \, z \in \Sigma \} \eeastar
and consider the
evaluation map
$$
\sigma^k: \CF_1(\Sigma,M;\beta) \to H^{(k,0)}\, ; \, (J,(j,u),z) \mapsto \sigma^k(J,(j,u),z).
$$
We will interpret this map as a section of some vector bundle over $\CJ_\omega \times \CF_1(\Sigma,M;\beta)$.

For any given $(j,J)$ and $(u,z) \in \CF_1(M;\beta;k) $, consider the
vector space
$$
H^{(k,0)}_{(J,(j,u),z)}: =
Sym^{(k,0)}_{j_z,J_{u(z)}}(T_z\Sigma,T_xM)
$$
of dimension $2n = \dim M$ and the vector bundle
$$
H^{(k,0)}_1: = \bigcup_{(J,(j,u),z)}
H^{(k,0)}_{(J,(j,u)z)}
$$
over $\CJ_\omega \times \CF_1(M;\beta;k) $ with
$H^{(k;0)}_{(J,(j,u),z)}$ as its fiber
where the union is taken for all $(J,(j,u,z)$.

Then the following lemma is immediate by definition.

\begin{lem} \label{lem:sigma} The map
$$
\sigma^k: \CJ_\omega \times \CF_1(M;\beta;k) \to H^{(k,0)}_1(\Sigma \times M)
$$
defined by $\sigma(J,(j,u),z) = \pi_{(j,J)}^{hol}(d^{k}u(z))$
is a smooth section of the vector bundle
$$
H^{(k,0)}_1 = H^{(k,0)}_1(\Sigma \times M) \to \CJ_\omega \times \CF_1(M;\beta).
$$
\end{lem}

\begin{rem} It is crucial in the Fredholm analysis that the section $\sigma^k$ can
be defined on the space of smooth maps,
not just on the moduli space of $J$-holomorphic maps.
\end{rem}

We introduce the standard bundle \be \CH'' = \bigcup_{((j,u),J)}
\CH''_{((j,u),J)}, \quad \CH''_{((j,u),J)} =
\Omega_{j,J}^{(0,1)}(u^*TM) \ee
and define a map $\Upsilon_k$ by
\be
\Upsilon_k (J,(j,u),z) = (\delbar(J,(j,u));\sigma^k(J,(j,u),z))
\ee
where we denote
$$
\delbar(J,(j,u)) : =  \delbar_{j,J}(u) = (du)^{(0,1)}_{j,J} = \frac{du + J du j}{2}
$$
We denote by
$$
\pi_1: \CJ_\omega \times \CF_1(M;\beta) \to \CJ_\omega \times \CF(M;\beta)
$$
the forgetful map of the marked point and consider the fiber product
$$
\CH''\times_{\pi_1} H_1^{(k,0)}
$$
of the two bundles,
$$
\CH'' \to \CJ_\omega \times \CF(M;\beta)
$$
and
$$
H^{(k,0)}_1\to \CJ_\omega  \times \CF_1(M;\beta) \to \CJ_\omega \times
\CF(\Sigma,M;\beta).
$$
More explicitly, we have
\beastar
&{}& \CH''\times_{\pi_1} H^{(k,0)}_1\\
& : = & \Big \{ (\eta,\zeta_0;J,(j,u),z) \,
\Big| \, \eta \in \CH''_{(J,(j,u))}, \, \zeta_0 \in
H^{(k,0)}_{(J,(j,u),z)}, \, (J,(j,u),z) \in \CF_1(M;\beta) \Big \}.
\eeastar
We regard the fiber product as a bundle over $\CJ_\omega \times
\CF_1(M;\beta) $ whose fiber at $(J,(j,u),z)$ is given by
$$
\CH''_{(J,(j,u))} \oplus H^{(k,0)}_{(J,(j,u),z)}.
$$
Then $\Upsilon_k$ defines a smooth map
$$
\Upsilon_k : \CJ_\omega \times \CF_1(M;\beta) \to \CH''\times
H_1^{(k,0)}
$$
which becomes a smooth section of this vector bundle.

One can generalize the above discussion by considering arbitrary finite
number of marked points and holomorphic $n$-jets, not just $\sigma^n$.

Let $\Sigma$ be a closed Riemann surface. We denote by $\widetilde{\Conf}_k(\Sigma)
\subset \Sigma^k$ the set of $k$ ordered distinct points on $\Sigma$, and
$$
\widetilde{\Conf}(\Sigma) = \bigcup_{k=0}^\infty \widetilde{\Conf}_k(\Sigma).
$$
For each given $k$ distinct points $\vec z = \{z_1, \cdots, z_k\}$,
we consider the decoration of integers $n_i$ assigned at $z_i$'s. We
call $k$ the \emph{length} of the configuration $\vec z \in
\widetilde{\Conf}(\Sigma)$. We denote $\vec n = \{n_1, \cdots,
n_k\}$ and $K = \{1, \cdots, k\}$. For given $k \leq k' $ with $K' =
\{1, \cdots, k'\}$, we decompose
$$
K' = K \cup (K' \setminus K).
$$
\begin{defn}\label{def:order} Consider the set of pairs $(k; \vec n)$ with
$\vec n \in \Z^k$
where $k = \operatorname{leng}(\vec n)$.
We say $(k';\vec n') < (k; \vec n)$ if
$$
k \leq k' , \mbox{ and } \quad n_i \leq n_i' \, \mbox{ for all } \,
i \in K \subset K'
$$
\end{defn}

\begin{rem} This definition of partial order is consistent with the
lower semi-continuity of the ramification degree under the limit of
a sequence of $J$-holomorphic maps in $C^\infty$-topology.
\end{rem}

We denote
$$
\CF_k(M;\beta) = \{((j,u), \vec z) \mid [u] = \beta, \, \vec z = (z_1, \cdots, z_k) \}.
$$
and define
$$
\widetilde \CM_{g,k}(M;\beta;\vec n)  = \{((j,u),z) \in \CF_k(M;\beta)  \mid \delbar_{(j,J)}u = 0, \,
\deg_{z_i} = n_i, \, \Crit(u) \supset \vec z \}.
$$
Then since the cardinality of and the degrees of ramification points of
a pseudo-holomorphic map are finite (see \cite{mcduff}, \cite{sikorav} for the proof),
we immediately have

\begin{lem} Let $J$ be any almost complex structure.
Denote by $\widetilde \CM^{(k;\vec n)}_g(M,J;\beta)$
the image of the forgetful map
$$
\frak{forget}_{(k;\vec n)} : \widetilde \CM_{g,k}(M,J;\beta;\vec n) \to
\widetilde \CM_g(M,J;\beta).
$$
Then we have the decomposition
\be\label{eq:decompose}
\widetilde \CM_g(M,J;\beta) = \bigcup_{(k;\vec n)} \widetilde \CM^{(k;\vec n)}_g(M,J;\beta).
\ee
\end{lem}

We call the $\widetilde \CM^{(k;\vec n)}_g(M,J;\beta)$ the
\emph{$(k;\vec n)$-stratum} of $\widetilde \CM_g(M,J;\beta)$. For
general $J$, the union in \eqref{eq:decompose} may not be a finite
union and the strata in $\widetilde \CM_g(M,J;\beta)$ may not be smooth.

We also define the union
$$
\widetilde \CM^{(k;\vec n)}_{g,\leq}(M,J;\beta) = \bigcup_{(k';\vec n') \le (k;\vec n)}
\widetilde \CM^{(k';\vec n')}_g(M,J;\beta)
$$
which is the closure of $\widetilde \CM^{(k;\vec n)}_g(M,J;\beta)$ in
$\widetilde \CM_g(M,J;\beta)$ in $C^\infty$ topology.

The main purpose of the present paper is to analyze the structure of
this decomposition and to establish certain stratawise
transversality for a generic choice of $J$.

\section{Higher jet evaluation transversality}
\label{sec:stratawise}

In this section, we first formulate the precise version of stratawise
transversality of higher jet evaluation maps. Then we will prove the transversality imitating the
proof of the 1-jet evaluation transversality Zhu and the present author gave in
\cite{oh-zhu}.

\subsection{Statement}

Denote by $\pi_k: \CJ_\omega \times \CF_k(M;\beta) \to \CJ_\omega \times \CF(M;\beta)$
the forgetful map and consider the fiber product
$
\CH'' \times_{\pi_k} \prod_{i=1}^k J^{n_i}_{hol}.
$
For each given $(k;\vec n)$, we consider a section
$$
\Upsilon_k^{\vec n} = \CJ_\omega \times \CF_k(M;\beta) \to
\CH'' \times_{\pi_k} \prod_{i=1}^k J^{n_i}_{hol}
$$
given by
\be\label{eq:Upsilonmn}
\Upsilon_k^{\vec n}(J,(j,u), \vec z) = (\delbar(J,(j,u)) ; (j^{n_1}_{hol}(z_1),\cdots,
j^{n_k}_{hol}(z_k))).
\ee
Denote by $o_E$ to be the zero section of any vector bundle $E$.
The following lemma immediately follows from the definition of $\Upsilon_k^{\vec n}$.

\begin{lem} For given $(k;\vec n)$, we have
\be \widetilde \CM_{g,k}(M;\beta;\vec n) = (\Upsilon_k^{\vec
n})^{-1}\left(o_{\CH'' \times_{\pi_k} \prod_{i=1}^k J^{(n_i,0)}}\right) \ee
and \be \widetilde \CM^{(k;\vec n)}_{g,\le}(M;\beta;k) =
\frak{forget}_{(k;\vec n)}\left(\widetilde \CM_{g,k}(M;\beta;\vec
n)\right). \ee
\end{lem}

This leads us to study the transversality property of
$\Upsilon_k^{\vec n}$. with respect to the zero section
$o_{\CH'' \times_{\pi_k} \prod_{i=1}^k J^{(n_i,0)}}
\subset \CH'' \times_{\pi_k} \prod_{i=1}^k J^{(n_i,0)}$.

The following is the main theorem we prove in this section.

\begin{thm}\label{thm:main} There exists a subset $\JJ_\omega^{ram} \subset \JJ_\omega$
such that for any $J \in \JJ_\omega^{ram}$ and $\beta \in H_2(M)$
and $g \in \N$, the linearization
$$
D_{(J,(j,u),\vec z)}\Upsilon_k^{\vec n}: T_J\CJ_\omega \times T_{((j,u),\vec z)}\CF_k(M,\beta) \to
\CH''_{(J,(j,u))}\oplus \bigoplus_{i=1}^k J^{n_i}_{hol,(J,(j,u),\vec z)}
$$
is surjective at all
$(J,(j,u),\vec z) \in (\Upsilon_k^{\vec n})^{-1}\left(o_{\CH'' \times_{\pi_k} \prod_{i=1}^k J^{n_i}_{hol}}\right)$.
In particular the set
$$
(\Upsilon_k^{\vec n})^{-1}\left(o_{\CH''\times_{\pi_k} \prod_{i=1}^k J^{n_i}_{hol}}\right)
\subset \CJ_\omega \times \CF_k(M,\beta)
$$
is a submanifold of $\widetilde \CM_{g,k}(M;\beta)$ of codimension
$
\sum_{i=1}^k 2n n_i.
$
\end{thm}

\subsection{Proof}

In this subsection, we give the proof of Theorem \ref{thm:main}. The
scheme of the proof is a generalization of the one used for the
1-jet transversality in \cite{oh-zhu} to the higher jets, which
however requires more sophisticated choice of function spaces in the
proof.

\medskip
\begin{proof}[Proof of Theorem \ref{thm:main}]
We consider the smooth section
$$
\Upsilon_k^{\vec n} : \CJ_\omega \times \CF_{g,k}(M,\beta) \to \CH'' \times_{\pi_k} \prod_{i=1}^k J^{(n_i,0)}
$$
defined by
$$
\Upsilon_k^{\vec n}(J,(j,u),z) = \left(\delbar(J,(j,u)), (j_{hol}^{n_1}(J,(j,u),z_1), \cdots, j_{hol}^{n_k}(J,(j,u),z_k))
\right)
$$
where $z = \{z_1, \cdots, z_k\} \in \widetilde{\Conf}(\Sigma)$.

The linearization map of $j^{n_i}_{hol}$ is the direct sum
$$
D_{(J,(j,u),z)}j^{n_i}_{hol} = \bigoplus_{\ell =1}^{n_i} D_{(J,(j,u),z_i)}\sigma^{\vec n;i}_\ell
$$
where $\sigma^{\vec n;i}_\ell : \CF_k(M,\beta) \to
H^{(\ell,0)}_{J,(j,u),z_i}$ is the evaluation of the $\ell$-th
holomorphic derivative at $z_i$ for $\ell \leq n_i$. And the map \be
D_{(J,(j,u),z)}\sigma^{n_i}_\ell  : T_J\CJ_\omega \times
T_{((j,u),\vec z)}\CF_k(M,\beta) \to H^{(k;0)}_{J,(j,u),z_i} \ee
is given by
\be\label{eq:DJjuzsigma}
D_{(J,(j,u),z)}\sigma^{n_i}_\ell(B,(b,\xi),\vec v) =
(D_{J,(j,u)}\delbar(B,(b,\xi)),
D_{(J,(j,u),\vec z)}(\sigma^{\vec n;i}_\ell)(B,(b,\xi),v_i)).
\ee
Here we have
\be\label{eq:Dsigma}
D_{(J,(j,u),\vec z)}(\sigma^{\vec n;i}_\ell)(B,(b,\xi),\vec v)) \nonumber \\
= D_{J,(j,u)}\sigma^{\vec n;i}_\ell(B,(b,\xi))(\vec z) + \nabla_{v_i} (\sigma^\ell(u))(z_i)
\ee
for $B\in T_J\CJ_{\o}, b\in T_j\CM(\S), v_i\in T_{z_i}\S$ and $\xi \in T_u\CF(M;\b)$.

Some remarks concerning the necessary Banach manifold set-up of the map
$\Upsilon$ are now in order :
\begin{enumerate}
\item To make evaluating $j^{k} u$ at a point $z \in
\Sigma$ make sense, we need to take at least $W^{k+1,p}$-completion
with $p > 2$ of $\CF_1(\b;k)$ at $z$ so that $j^{k}(u)$ lies in
$W^{1,p}$ at $z$ which is then continuous at $z$. We actually need
to take $W^{N,p}$-completion of $\CF_1(\b;k)$ with $N= N(\beta,k)$
sufficiently large so that the section $\Upsilon$ is differentiable
and that Sard-Smale theorem can be applied.
\item We provide $\CH''$ with the topology of a $W^{N,p}$ Banach bundle.
\item We also need to provide some Banach manifold structure on $\CJ_\omega$.
We can borrow Floer's scheme \cite{floer} for this whose details we
refer readers thereto.
\end{enumerate}

We now complete the tangent space $T_{((j,u),z)}\CF_k(M,\beta)$ by
the $W^{N,p}$-norm with $N$ sufficiently large so that $N$ is
at least \be\label{eq:N} N \geq \max\{3,
\max_i \{n_i + 2\mid i =1, \cdots k \}\}. \ee The choice of $N$ will
vary depending only on the homology class $\beta$ and the genus $g$.
Recall that if $u$ is in $W^{N,p}$, $D_{J,(j,u)}\sigma^{\vec
n;i}(B,(b,\xi))$ and $\nabla_{v_i} (\sigma(u))$ are in
$W^{N-n_i-1,p}$ with $N-n_i-1 \geq 1$. Therefore their evaluations
at $z_i$ are well defined since any $W^{1,p}$-map is continuous.

At fixed $(J, (u,j),\vec z)$ where we do linearization of $\Upsilon_k^{\vec n}$,
we will write
\beastar
\Omega^0_{N,p}(u^*TM) & := & W^{N,p}(u^*TM) = T_u \CF^{N,p}(\Sigma,M;\beta)\\
\Omega^{(0,1)}_{N-1,p}(u^*TM) & := & W^{N-1,p}\left(\Lambda_{(j,J)}^{(0,1)}(u^*TM)\right)
\eeastar
for the simplicity of notations.

To prove Theorem \ref{thm:main}, we need to verify that
at each given point $(J,(j,u),\vec z) \in
(\Upsilon_k^{\vec n})^{-1}\left(o_{\CH'' \times_{\pi_k} \prod_{i=1}^k J^{n_i}_{hol}}\right)$,
the system of equations
\bea%
D_{J,(j,u)}\delbar(B,(b,\xi)) & = & \gamma \label{eq:Dbar}\\
D_{J,(j,u)}(\sigma^{\vec n;i})(B,(b,\xi))(z_i)
+ \nabla_{v_i} (\sigma^{\vec n;i}_\ell(u)) & = & \zeta_{i;\ell} \nonumber\\
\mbox{for } \, \ell = 1, \cdots, n_i, \, i =1, &\cdots,& k
\label{eq:DJueta-}
\eea%
has a solution $(B,(b,\xi),\vec v)$ for each given data
$$
\gamma \in \Omega^{(0,1)}_{N-1,p}(u^*TM), \quad \zeta_{i;\ell} \in H^{(\ell,0)}_{J,(j,u),z_i}
$$
with $1 \leq \ell \leq n_i$ and $1 \leq i \leq k$. It will be enough
to consider the triple with $b = 0$ and $\vec v=0$ which we will
assume from now on.

We now compute the linearization $D_{(J,(j,u),z)}\sigma^{\vec n;i}_\ell(B,(b,\xi),\vec
v))$. We first recall that $\sigma^{\vec n;i}_\ell$ defines a
section of the pull-back of the vector bundle $H^{(\ell,0)}_1 \to
\CF_1(M;\beta)$ to $\CF_k(M;\beta)$ via the forgetful map
$\CF_k(M;\beta)\to \CF_1(M;\beta)$, and the linearization is meant
to be the covariant linearization of the section. Note that
computation of this linearization is local near $z_i \in \Sigma$,
and so we can use coordinate calculations at $z_i$ and $u(z)$ as in
section \ref{sec:jet}. By $J$-complex linearity of $\nabla$ and the
\emph{vanishing at $z_i$ of the $n_i$-jet $j^{n_i}u(z_i)$}, it is
easy to see that we have
$$
D_{(J,(j,u),\vec z)}(\sigma^{\vec n;i}_\ell)(0,(0,\xi),0))(\vec z)
=  \pi_{hol}\left((\nabla_{du})^{n_i} \xi(z_i)\right)
$$
as long as $1 \leq \ell \leq n_i$.

\begin{rem} We would like to point out that for a general map $u$
the formula for $D_{(J,(j,u),\vec z)}(\sigma^{\vec
n;i}_\ell)(0,(0,\xi),0))(\vec z)$ involve products of $\nabla_{du}^k
\xi$ and $\nabla^j u$ with $0 \leq k, j \leq n_i$ and $k+j = n_i$ in
addition to $\pi_{hol}\left((\nabla_{du})^{\ell} \xi(z_i)\right)$.
Those terms with $1 \leq j \leq n_i$ will vanish by the condition
$j^{n_i}u = 0$.
\end{rem}

Since $u$ is $(j,J)$-holomorphic, it also follows that
$$
\pi_{hol}\left((\nabla_{du})^{\ell}\xi(z_i)\right) =
(\nabla_{du}')^{\ell} \xi(z_i)
$$
where $\nabla_{du}' = \pi_{hol}\nabla_{du}$.

 Now we study solvability of
\eqref{eq:Dbar}-\eqref{eq:DJueta-} by applying the Fredholm
alternative. We regard
$$
\Omega^{(0,1)}_{N-1,p}(u^*TM) \times \prod_{i=1}^k J^{n_i}_{hol;(J,(j,u),z_i)}
$$
as a Banach space with the norm
$$
\|\cdot\|_{N-1,p} + \sum_{i=1}^k\sum_{\ell=1}^{n_i} |\cdot|_{z_i;\ell}
$$
where $|\cdot|_{z_i;\ell}$ any norm induced by an inner product on
the $2n$-dimensional vector space $H^{(\ell,0
)}_{(J,(j,u),z_i)}\cong \C^n$.

For the clarification of notations, we denote any natural pairing
$$
\Omega^{(0,1)}_{N-1,p} \times \left(\Omega^{(0,1)}_{N-1,p}\right)^*\to \R
$$
by $\langle \cdot, \cdot
\rangle$ and the inner product on $H^{(\ell,0)}_{(J,(j,u),z_i)}$ by $(\cdot,
\cdot)_{z_i}$.

\begin{rem} \label{rem:3pversus2p}
We emphasize that for the map
\be\label{eq:z0v}
\vec z \mapsto D_{J,(j,u)}\sigma^{\vec n;i}_\ell(B,(b,\xi))(z_i)
\ee
to be defined as a continuous map to $H_{z_i}^{(\ell,0)}$, the map $u$ must
be at least $W^{N,p}$ for $N \geq \max\{n_i + 1\}$ near each $z_i$.
On the other hand, as it will be clear from the discussion in section \ref{sec:removal}
we need to reduce the regularity of the completed Sobolev space from $W^{N,p}$ to
$W^{n_i+1,p}$ \emph{locally near at} each $z_i$ respectively.
\end{rem}

Due to this remark and the fact that $n_i$ vary over $i$, we will first consider the problem on
a space with regularity weaker than $W^{N,p}$ but stronger than $W^{n_i+1,p}$ \emph{locally
near at} $z_i$. In the end of
the proof, we will establish solvability of \eqref{eq:Dbar}-\eqref{eq:DJueta-}
on $W^{N,p}$ by applying an elliptic regularity result
of the equation \eqref{eq:Dbar}.

To overcome the fact that $n_i$ vary over $i$,
we fix a choice of cut-off functions $\chi = \sum_i \chi_i$ so that
$\supp \chi_i \in D_i$ and $\chi_i \equiv 1$ on $V_i \subset D_i$, and
$D_i$ are disjoint from one another. Define a
norm
\be\label{eq:mixednorm}
\|\xi\|_{\vec n + \vec 1;\vec z} = \sum_{i=1}^k \|\chi_i \xi\|_{n_i+1,p}
\ee
and the space
$$
\Omega^{0}_{\vec n + \vec 1 ;\vec z}(u^*TM) = \{\xi \in W^{2,p} \mid
\|\xi\|_{\vec n + \vec 1;\vec z} < \infty\}
$$
as the completion of $\Omega^{(0,1)}(u^*TM)$ with respect to
the norm $\|\cdot \|_{\vec n + \vec 1;\vec z}$.

With this definition, we first consider the equations
\eqref{eq:Dbar} and \eqref{eq:DJueta-}
for $\xi \in \Omega^{0}_{\vec n+\vec 1;\vec z}(u^*TM)$ when $\gamma$ lies
in $\Omega^{(0,1)}_{\vec n;\vec z}(u^*TM)$.
We will derive the solvability
for the case  $\Omega^{(0,1)}_{N-1,p}(u^*TM)$ afterwards
by applying the elliptic regularity of the equation \eqref{eq:Dbar}.

\begin{prop}\label{prop:b=v=0}
The map
\eqref{eq:DJjuzsigma}
$$
D\Upsilon_k^{\vec n}: T_J \JJ_\omega \times \Omega^{0}_{\vec n + \vec 1;\vec z}(u^*TM) \to
\Omega^{(0,1)}_{\vec n;\vec z}(u^*TM) \times \prod_i J^{n_i}_{hol;J,(j,u),z_i} := \CB
$$
restricted to the elements of the form $(B,(0,\xi),0)$ is onto at
any $(J,(u,j),\vec z)$ that lies in $(\Upsilon_k^{\vec
n})^{-1}(o_{\CH^{''}}\times_{\pi_k} o_{\prod_i J^{n_i}_{hol}}) =
\widetilde\MM_{g,k}(M,\beta;\vec n)$.
\end{prop}

\begin{proof}
To prove the surjectivity, we will prove that the image of
$D_{J,(j,u),\vec z}\Upsilon_k^{\vec n}$ is dense and closed in
$\CB$.

We start with the denseness. Let $(\eta, \vec \zeta) \in
(\Omega^{(0,1)}_{N-1,p}(u^*TM))^* \times \prod_{i=1}^k
J^{n_i}_{hol;(J,(j,u),z_i)}$ for $\vec \zeta = (\zeta_{i;\ell})$
such that \be\label{eq:0} \langle
D_{J,(j,u)}\delbar_{j,J}(B,(0,\xi)), \eta \rangle +
\sum_{i=1}^k\sum_{\ell=1}^{n_i} \left(D_{J,(j,u)}\sigma^{\vec
n;i}_\ell(B,(0,\xi))(z_i), \zeta_{i;\ell}\right)_{z_i}= 0 \ee for
all $\xi \in \Omega^0_{\vec n +\vec 1,p}(u^*TM)$ and $B$. It will be
enough to consider smooth $\xi$'s in our consideration of
\eqref{eq:0} since $\Omega^0(u^*TM) \hookrightarrow \Omega_{\vec
n;\vec z}^0(u^*TM)$ is dense. Under this assumption, we would like
to show that $\eta = 0 = \zeta_{i,\ell}$.

By the above discussion on $D_{J,(j,u)}\delbar(B,(0,\xi))$
and $D_{J,(j,u)}\sigma^{\vec n;i}_\ell(B,(0,\xi))(z_i)$, \eqref{eq:0} is equivalent to
\be\label{eq:simplifiedcoker}%
\langle D_u\delbar_{(j,J)} \xi+ \frac{1}{2}B\circ du\circ j, \eta \rangle
+ \sum_{i,\ell} \langle (\nabla_{du}')^\ell\xi, \delta_{z_i} \zeta_{i,\ell}\rangle =0
\ee%
for all $B$ and $\xi$ of $C^{\infty}$ where $\delta_{z_0}$ is the
Dirac-delta function.

\begin{lem}\label{lem:highd}
Suppose $j^{n_i}u(z_i) = 0$. Then we have
\be\label{eq:nabladu'}
(\nabla_{du}')^{\ell} \xi (z_i)= \del^{\ell} \xi (z_i)
\ee
for all $1 \leq \ell \leq n_i$ and $i =1, \cdots, k$
where $\del$ is the Dolbeault differential with respect to $(j,J)$.
\end{lem}

Taking $B=0$ in \eqref{eq:0} and applying Lemma \ref{lem:highd}, we obtain
\be \label{eq:coker} %
\langle D_u\delbar_{(j,J)} \xi, \eta \rangle +  \sum_{i,\ell}\langle
\del^{\ell} \xi, \delta_{z_i} \zeta_{i,\ell}\rangle = 0 \quad \mbox{
for all $C^\infty$ section $\xi$}.
\ee %
Therefore by definition of the distribution derivatives, $\eta$ satisfies
$$
(D_u\delbar_{(j,J)})^\dagger \eta + \sum_{i,\ell} (\del^\dagger)^{\ell}(\delta_{z_i} \zeta_{i;\ell}) = 0
$$
as a distribution, i.e.,
$$
(D_u\delbar_{(j,J)})^\dagger \eta = - \sum_{i,\ell} (\del^\dagger)^{\ell}(\delta_{z_i} \zeta_{i;\ell})
$$
where $(D_u \delbar_{j,J})^\dagger$ is the formal adjoint of $D_u
\delbar_{j,J}$ whose symbol is the same as $D_u\del_{j,J}$ and so
 first order differential operator. We also recall that the adjoint
$\del^\dagger$ of $\del$ is a first order elliptic operator which has the same
principal symbol as $- \delbar$.
Since $\supp (\del^\dagger)^{\ell}(\delta_{z_i}
\zeta_{i,\ell}) \subset \{z_1, \cdots, z_k\}$, we have $(D_u\delbar_{(j,J)})^\dagger
\eta = 0$ on $\Sigma \setminus \{z_1, \cdots, z_k\}$ as a distribution. Then by
the elliptic regularity (see Theorem 13.4.1 \cite{Ho} for example),
$\eta$ must be smooth on $\Sigma \setminus \{z_1, \cdots, z_k\}$.

On the other hand, by setting $\xi = 0$ in
\eqref{eq:simplifiedcoker},  we get
\be \label{eq:1/2B} \langle
B\circ du\circ j, \eta \rangle=0
\ee
for all $B\in T_{J}
\CJ_\omega$. From this identity, standard argument from
\cite{floer}, \cite{mcduff} shows that $\eta=0$ in a small
neighborhood of any somewhere injective point in $\Sigma \setminus
\{z_0\}$. Such a somewhere injective point exists by the hypothesis
of $u$ being somewhere injective (see Notation in the introduction)
and the fact that the set of somewhere injective points is open and
dense in the domain under the hypothesis (see \cite{mcduff}). Then
by the unique continuation theorem, we conclude that $\eta = 0$ on $
\Sigma \backslash \{z_1, \cdots, z_k\}$ and so the support of $\eta$
as a distribution on $\Sigma$ is contained at the subset $\{z_1,
\cdots, z_k\}$ of $\Sigma$.

We will postpone the proof of the following lemma till section
\ref{sec:removal}.

\begin{lem} \label{lem:eta=0}  $\eta$ is a distributional solution of
$(D_u \delbar_{j,J})^\dagger \eta = 0$ on $\Sigma$
and so continuous.
In particular, we have $\eta = 0$ in $\left(\Omega^{(0,1)}_{(N-1,p)}(u^*TM)\right)^*$.
\end{lem}

Once we know $\eta = 0$, the equation \eqref{eq:0} is reduced to
\be\label{eq:simple0}
\sum_{i=1}^k \sum_{\ell =1}^{n_i}\left(D_{J,(j,u)}\sigma(B,(0,\xi))(z_i), \zeta_{i;\ell}\right)_{z_i} = 0
\ee
It remains to show that $\zeta_{i;\ell} = 0$. By considering $\xi$ supported in
a disjoint union of small neighborhoods of $z_i$'s, we obtain
$$
\sum_{\ell=1}^{n_i} \left(D_{J,(j,u)}\sigma(B,(0,\xi))(z_i), \zeta_{i;\ell}\right)_{z_i} = 0
$$
for all such $\xi$. Therefore to show $\zeta_{i;\ell} = 0$ for all $1 \leq \ell \leq n_i$
at each $i$, we have only to
show that the image of the evaluation map
$$
\xi \mapsto \sum_{\ell=1}^{n_i} D_{J,(j,u)}\sigma^{\vec n;i}_\ell (0,(0,\xi))(z_i) =
\sum_{\ell=1}^{n_i} \del^{\ell} \xi(z_i)
$$
is surjective onto $J^{n_i}_{hol;J,(j,u),z_i}$. First of all, we note that the
$\ell$-th holomorphic jets of \emph{smooth} $\xi$ for $\ell =1, \cdots, n_i$
are functionally independent and can be chosen freely separately. This reduces the
surjectivity question order by order.

We focus on the $\ell$-th jet for each given $\ell =1, \cdots, n_i$.
To show this surjectivity, we need to prove the existence of $\xi$ satisfying
\be\label{eq:delunell}%
\del^{\ell} \xi (z_i) = \zeta_{i;\ell}
\ee%
at $z_i$ for any given $\zeta_{i;\ell} \in H^{(\ell,0}_{z_i}$.
We can multiply a cut-off function $\chi$ to $\zeta_{i;\ell}$ with
$\chi \equiv 1$ to make $\zeta(z): = \chi(z) \zeta_{i;\ell}$
and we may assume $\zeta$ is supported in a
sufficient small neighborhood around $z_i$. If we write
$$
\sigma^\ell(J,(j,u),z_i) = \vec a \cdot dz^{\otimes \ell}
$$
with $\vec a(z) = \frac{\del^{\ell} u}{\del z^{\ell}}(z_i)$ and
$$
\zeta_{i;\ell}(z_i) = b(z_i)\cdot dz^{\otimes \ell},
$$
\eqref{eq:delunell} is reduced to the equation
$$
\frac{\del^{\ell} \xi}{\del z^{\ell}}(z_i) = b(z_i).
$$
By simply integrating this equation, we solve this equation in some
neighborhood around $z_0$, which in turn solves $\del^{\ell}\xi(z_i)
= \zeta_{i,\ell}$. This finishes the proof of existence of a
solution to (\ref{eq:delunell}) and hence to \eqref{eq:DJueta-}.
This then proves that the image of \eqref{eq:DJjuzsigma} with $v=0$
is \emph{dense} in \be\label{eq:OmegaJ} \Omega^{(0,1)}_{\vec n;\vec
z}(u^*TM) \times \prod_i^k J^{n_i}_{hol;(J,(j,u),z_i)} \ee as a map
from $T_J\JJ_\omega \times \Omega^0_{\vec n+\vec 1;\vec
z,p}(u^*TM)$. On the other hand by the elliptic regularity, it
follows that for any fixed $J$, the image of $D_{J,(j,u)}\delbar$
from $T_{(u,j)}\CF_{g,k}(\beta;\vec n)$ is \emph{closed} in
$\Omega^{(0,1)}_{\vec n;\vec z}(u^*TM)$. We also note that
$J^{\ell}_{hol;J,(j,u),z_i}$ is a finite dimensional vector space.
These imply that the image of \eqref{eq:DJjuzsigma} is closed in
\eqref{eq:OmegaJ}. Therefore the map \eqref{eq:delunell} is onto
$\Omega^{(0,1)}_{\vec n;\vec z}(u^*TM) \times \prod_{i=1}^k
J^{n_i}_{hol;J,(j,u),z_0}$ as a map from $T_J\JJ_\omega \times
\Omega^0_{\vec n+\vec 1;\vec z}(u^*TM)$. This finishes the proof of
Proposition \ref{prop:b=v=0}.
\end{proof}

Now we finally go back to the study of \eqref{eq:Dbar}-\eqref{eq:DJueta-}
for the case of
$$
\gamma \in \Omega^{(0,1)}_{N-1,p}(u^*TM) \subset
\Omega^{(0,1)}_{1,p}(u^*TM).
$$
We recall $N \geq 3$. By the above analysis of the linearization
$D\Upsilon_k^{\vec n}$ on $\Omega^{(0,1)}_{\vec n;\vec z}$, we can
find a solution $(B,(0,\xi),0)$ of
\eqref{eq:Dbar}-\eqref{eq:DJueta-} with $B \in T_J\JJ_\omega$ and
with \emph{$\xi$ as an element in $\Omega^0_{\vec n+\vec 1;\vec
z}(u^*TM)$} for any $\vec \zeta = (\zeta_{i;\ell})$. Applying
elliptic regularity to the equation \eqref{eq:Dbar}, we derive that
$\xi$ indeed lies in $W^{N,p}$ if $\gamma \in W^{N-1,p}$, and hence
in $\Omega^0_{N,p}(u^*TM)$.

Therefore the map $\eqref{eq:DJjuzsigma}$ is onto, i.e., $\Upsilon_k^{\vec n}$ is transverse to
the submanifold
$$
o_{\CH'' \times_{\pi_k} \prod_{i=1}^kJ^{n_i}_{hol}}\subset \CH''
\times_{\pi_k} \prod_{i=1}^kJ^{n_i}_{hol}.
$$
This finishes the proof of Theorem \ref{thm:main}.
\end{proof}

\medskip

It follows from definition that
$$
\widetilde \CM_{g,k}(M,\beta;\vec n)
= (\Upsilon_k^{\vec n})^{-1}\left( o_{\CH'' \times_{\pi_k} \prod_{i=1}^kJ^{n_i}_{hol}}\right)
$$
and we have the natural projection
$$
\pi:\widetilde \CM_{g,k}(M,\beta;\vec n) \to \CJ_\omega.
$$
Then we have
$$
\widetilde \CM_{g,k}(M,J;\beta;\vec n)=\widetilde \CM_{g,k}(M,\beta;\vec n)
\cap \pi^{-1}(J).
$$
We denote
$$
\CJ_\omega^{g,k,\vec n} = \mbox{the set of regular values of $\pi$}
$$

An immediate corollary of this proposition and the discussion in section \ref{sec:Fredholm} is

\begin{cor} For any $J \in \CJ_\omega^{g,k,\vec n}$, $\widetilde \CM_{g,k}(M,J;\beta;\vec n)$
is a smooth manifold of $\widetilde \CM_{g,k}(M,J;\beta)$ of
codimension $(\sum_{i=1}^k 2nn_i)$.
\end{cor}
\begin{proof} Here each $2nn_i$ comes from the vanishing of $n_i$ derivatives
at a marked point $z_i$ and $-2$ comes from the location of marked point
$z_i$ in $\Sigma$.
\end{proof}

Now we set
$$
\CJ_\omega^{ram} = \bigcap_{g\in \Z_{\geq 0}}\bigcap_{k \in \N}\bigcap_{\vec n \in \N^k} \CJ_\omega^{g,k,\vec n}.
$$
Obviously $\CJ_\omega^{ram} $ is a subset of $\subset \CJ_\omega$ of second category.

\section{Removal of singularity : Proof of Lemma \ref{lem:eta=0}}
\label{sec:removal}

In this section, we prove Lemma \ref{lem:eta=0}. Our primary goal
is to prove
\be\label{eq:tildexi} \langle D_u\delbar_{(j,J)} \xi,
\eta \rangle = 0
\ee
for all smooth $\xi \in \Omega^0(u^*TM)$,
i.e., $\eta$ is a distributional solution of
$(D_u\delbar_{(j,J)})^\dagger \eta = 0$ \emph{on the whole
$\Sigma$}, not just on $\Sigma \setminus \{z_1, \cdots, z_k\}$ which was shown
in section \ref{sec:stratawise}.
In addition, $\eta$ is a continuous linear functional on
$\Omega^{(0,1)}_{N-1,p}(u^*TM)$.

We start with \eqref{eq:coker}, which is
\be\label{eq:Tgt}%
\langle D_u\delbar_{(j,J)} \xi, \eta \rangle
+ \sum_{i,\ell} \langle \del^{\ell} \xi, \delta_{z_i} \zeta_{i,\ell}\rangle =0
\ee%
for all $\xi$ of $C^{\infty}$. We first  simplify the expression of
the pairing $\langle D_u\delbar_{(j,J)} \xi, \eta \rangle$ knowing
that $\supp \eta \subset \{z_1,\cdots,z_k\}$. Recall the well-known
computation \be\label{eq:formula} D_{J,(j,u)}\delbar(0,(0,\xi)) =
D_u \delbar_{j,J}\xi = (\nabla_{du} \xi)^{(1,0)}_{(j,J)} +
T^{(1,0)}_{(j,J)}(du, \xi) \ee with respect to a $J$-complex
connection $\nabla$ and its torsion tensor $T$. Here we denote
$$
T^{(1,0)}_{(j,J)}(du,\xi) = \frac{1}{2}\left(T(du,\xi) + J T(du \circ j, \xi)\right).
$$
From this it follows that
\be\label{eq:nablakT}
\nabla_{du}^{\ell}\left(T^{(1,0)}_{(j,J)}(du,\xi)\right) (z_i) = 0
\ee
for all $0 \leq \ell \leq n_i - 1$, provided $j^{n_i}u(z_i) = 0$.

Now we simplify the expression of $(\nabla_{du} \xi)^{(0,1)}_{(j,J)}$
in complex coordinates $z$ at $z_0$.
Let $x_0=u(z_0)$, and identify a
neighborhood of $z_0$ with an open subset of $\C$ and a neighborhood
of $x_0$ with an open set in $T_{x_0}M$. Then if we
identify $(T_{x_0} M,J_{x_0}) \cong \C^n$, we can write the
operator
\be\label{eq:nabladuxi}
(\nabla_{du} \xi)^{(0,1)}_{(j,J)}
= \delbar \xi + C \cdot\del \xi + D \cdot \xi,
\ee
where in a neighborhood of $z_0$, $\del, \, \delbar$ are the
standard Cauchy-Riemann operators on $\C^n$ and $C =C(x), \, D = D(x)$ are smooth pointwise
(matrix) multiplication operators whose coefficients depend only on
$M$ and $J$ and satisfies
\be\label{eq:CD}
C(x_0) = 0 =  D(x_0), \quad x_0 = u(z_0).
\ee
Adding \eqref{eq:formula} and
\eqref{eq:nabladuxi}, we can write
$$
D_{J,(j,u)}\delbar(0,(0,\xi)) = \delbar \xi + E\cdot \del \xi + F \cdot \xi
$$
(See \cite{sikorav}, \cite{oh-zhu}.) The following is a simple consequence of
the chain rule and \eqref{eq:nablakT} and \eqref{eq:CD}.

\begin{lem}\label{lem:nablak}
Suppose $j^{n_i} u(z_0) = 0$ and
Let $E = E(u(z)), \, F = F(u(z))$ in the above formula.
Then we have
\be\label{eq:nablakExi}
\nabla_{du}^k(E \cdot \xi )(z_0) = 0 = \nabla_{du}^k(F \cdot \xi)(z_0)
\ee
for all $0 \leq k \leq n_i$.
\end{lem}

Let $z$ be a complex coordinate centered at $z_0$ and  $(w_1,
\cdots, w_n)$ be the complex coordinates on $M$ regarded as
coordinates on a neighborhood of $u(z_0)$. We consider the standard
metric
$$
h = \frac{\sqrt{-1}}{2} dz d\bar z
$$
on a neighborhood $U$ of $z_0$ and with respect to the coordinates
$(w_1,...,w_n)$ we fix any Hermitian metric on $\C^n$.

We fix complex coordinates
satisfying \eqref{eq:complexatp0} at each $z_i$ and denote by
$\delbar$ the Dolbeault differential with respect to the complex coordinates
on the corresponding coordinate neighborhoods respectively.
We fix cut-off functions $\chi_i$ whose support
$\supp \chi_i$ is contained in $D_i$ a neighborhood of
$z_i$ respectively. We also assume that $D_i$'s are disjoint from one another.

By multiplying a cut-off function $\chi = \sum_i\chi_i$ to $\xi$,
the map $\delbar_\chi$ defined by
$$
\delbar_\chi(\xi) := \sum_i \delbar (\chi_i \xi)
$$
gives rise to a well-defined continuous operator from
$\Omega^0_{\vec n+\vec 1;\vec z}(u^*TM)$ to $\Omega^{(0,1)}_{\vec n;\vec z}$.

The following proposition will be crucial in our proof. Here our
choice of the above particular mixed Sobolev norm enters in the proof
in a crucial way similar as in the proof of Lemma 2.5 \cite{oh-zhu}.

\begin{prop} Let $\eta \in \left(\Omega^{(0,1)}_{\vec n;\vec z}(u^*TM)\right)^*$
be the distribution on $\Sigma$ obtained above.
Then for any smooth section $\xi$ of $u^*(TM)$, we have
$$
\langle D_u\delbar_{(j,J)}\xi , \eta \rangle = \langle \delbar_\chi \xi,
\eta \rangle.
$$
\end{prop}
\begin{proof} We have already shown that
$\eta$ is a distribution with $\supp \eta \subset \{z_1,\cdots,z_k\}$. By the
structure theorem on the distribution supported at a point $z_0$
(see section 4.5, especially p. 119, of \cite{gelfand}), we have
\be\label{eq:eta=P}
\eta = \sum_i P_i\left(\frac{\del}{\del s}, \frac{\del}{\del t}\right)(\delta_{z_i})
\ee
where $z = s + it$ is the given complex coordinates at $z_i$ and
$P_i\left(\frac{\del}{\del s}, \frac{\del}{\del t}\right)$ is a differential
operator associated with the polynomial $P_i$ of two variables.
Furthermore since $\eta \in (W^{n_i,p})^*$, the degree of $P_i$ must be less
than equal to $n_i-1$ : This is because the `evaluation
at a point of the $n_i$-th derivative of $W^{n_i,p}$ map does not
define a continuous functional on $W^{n_i,p}$.

By multiplying a cut-off function $\chi = \sum_i\chi_i$ and using the
support condition on $\eta$, we have
$$
\langle D_u\delbar_{(j,J)}\xi , \eta \rangle =
\sum_i \langle D_u\delbar_{(j,J)}(\chi_i \xi) , \eta \rangle.
$$
Therefore to prove the lemma, it is enough to prove
$$
\langle D_u\delbar_{(j,J)}(\chi_i \xi) , \eta \rangle = \langle \delbar(\chi_i \xi), \eta \rangle
$$
for each $i$. We now recall
$$
D_u\delbar_{(j,J)}\xi = \delbar \xi + E \cdot \del \xi + F\cdot \xi
$$
in coordinates at $z_i$ where $E$ and $F$ are zero-order matrix operators with $E(z_i) = 0 = F(z_i)$
satisfying \eqref{eq:nablakExi}.
Therefore by \eqref{eq:eta=P}, we derive
$$
\langle E \cdot \del (\chi_i\xi) + F \cdot (\chi_i \xi), \eta \rangle
=  \left\langle E \cdot \del (\chi_i\xi)
+ F\cdot (\chi_i \xi), P_i\left(\frac{\del}{\del s}, \frac{\del}{\del t}\right) \delta_{z_i}\right \rangle.
$$
By writing out
$$
P_i\left(\frac{\del}{\del s},\frac{\del}{\del t}\right)
= \sum_{0 \leq a+b \leq n_i -1} \vec a_{a,b} \frac{\del^{a+b}}{\del s^a\del t^b},
$$
integrating by parts and then applying Lemma \ref{lem:nablak}, we obtain
\beastar
&{}& \langle E \cdot \del (\chi_i \xi) + F \cdot(\chi_i \xi), \eta \rangle \\
& = &
\sum_{0\leq a+b \leq n_i-1} (-1)^{a+b} \left\langle \frac{\del^{a+b}}{\del s^a\del t^b}
\left(E \cdot \del (\chi_i\xi) + F \cdot (\chi_i \xi)\right), \vec a_{a,b} \delta_{z_i} \right\rangle \\
& = & \sum_{0 \leq a+b \leq n_1-1} (-1)^{a+b} \left(
\frac{\del^{a+b}}{\del s^a\del t^b} \left(E \cdot \del (\chi\xi) + F
\cdot (\chi\xi)\right)(z_0), \vec a_{a,b}\right)_{z_i} = 0. \eeastar
Therefore we obtain
$$
\langle D_u\delbar_{(j,J)}\xi , \eta \rangle = \langle \delbar_\chi \xi + E \cdot \del \xi + F\cdot \xi,
\eta \rangle = \langle \delbar_\chi \xi, \eta \rangle
$$
which finishes the proof.
\end{proof}

This lemma then implies that \eqref{eq:Tgt} is equivalent to
\be\label{eq:coker-simple}
\langle \delbar_\chi \xi, \eta \rangle + \sum_{i=1}^k\sum_{\ell=1}^{n_i}
\langle \del^{\ell} \xi, \delta_{z_i} \zeta_{i;\ell}
\rangle = 0 \quad
\mbox{for all $\xi$}.
\ee
Express  $\del^{\ell}\xi$ as \be\label{eq:delxiz}
\del^{\ell}\xi(z_i) = a_{i;\ell}(z)\, dz^{\otimes \ell} \ee in $D_i$
in coordinates with $a_{i;\ell}(z_i) \in \C^n$.
We decompose $\xi$ as
$$
\xi = \widetilde \xi(z) + \frac{1}{\ell !}\sum_i\sum_\ell \chi_i(z)
(z-z_i)^{\ell} a_{i;\ell}(z_i)
$$
by defining $\widetilde\xi$ by
$$
\widetilde \xi(z) = \xi(z) - \frac{1}{\ell !}\sum_i\sum_\ell
\chi_i(z) (z-z_i)^{\ell} a_{i;\ell}(z_i).
$$
Our choice of this
decomposition is dictated by the fact
\be \label{eq:delchi}
\del^{\ell}\left(\chi_i(z) (z-z_i)^{\ell} a_{i;\ell}(z_i)\right)
(z_i) = \ell ! \cdot a_{i;\ell}(z_i)\cdot dz^{\otimes \ell}. \ee Then
$\widetilde \xi$ is a smooth section on $\Sigma$, and satisfies
$$
\quad \del^{\ell} \widetilde \xi(z_i) = 0,
$$
and
$$
\delbar \widetilde \xi = \delbar{\xi} \quad \mbox{on $V_i$}
$$
for all $i, \, \ell$.
Therefore applying \eqref{eq:coker-simple} to $\widetilde \xi$ instead of $\xi$, we obtain
$$
\langle \delbar_\chi \widetilde \xi, \eta \rangle + \sum_{i,\ell} \langle \del^{\ell}
(\chi_i\widetilde \xi), \delta_{z_i} \zeta_{i;\ell} \rangle = 0.
$$
But we have
\be\label{eq:tildexi=xi}
\langle \delbar_\chi \widetilde \xi, \eta \rangle = \langle \delbar_\chi \xi, \eta \rangle
\ee
since $\delbar \widetilde \xi = \delbar{\xi}$ on $V_i$ and $\supp \eta \subset \{z_1, \cdots, z_k\}$.
Again using the support property $\supp \eta \subset \{z_1, \cdots, z_k\}$
and \eqref{eq:delxiz}, \eqref{eq:delchi},  we derive
\bea\label{eq:delxi-adz}
\langle \del^{\ell} \widetilde \xi, \delta_{z_i} \zeta_{i,\ell}\rangle & = &
\langle \del^{\ell}\xi, \delta_{z_i} \zeta_{i;\ell} \rangle -
\langle \del^{\ell}(\chi(z)(z-z_i)^{\ell} a_{i;\ell}(z_i)),
\delta_{z_i}\zeta_{i;\ell} \rangle \nonumber\\
& = & (\del^{\ell}\xi,\zeta_{i;\ell})_{z_i} - (a_{i;\ell}(z_i)dz^{\otimes \ell}, \zeta_{i;\ell})_{z_i} \nonumber\\
& = & (\del^{\ell}\xi(z_i)- a_{i;\ell}(z_i)dz^{\otimes \ell},\zeta_{i;\ell})_{z_i} = 0
\eea
where the equality next to the last comes from \eqref{eq:delxiz}.
Substituting \eqref{eq:tildexi=xi} and
\eqref{eq:delxi-adz} into \eqref{eq:coker-simple}, we obtain
$
\langle \delbar_\chi \xi, \eta \rangle = 0
$
and hence \eqref{eq:tildexi} follows.

Since \eqref{eq:tildexi} holds for all $\xi$, we have proved that
$\eta$ is a distributional solution of $(D_u\delbar_{j,J})^\dagger \eta = 0$ on
$\Sigma$. This finishes the proof of the first part of the lemma.

This then implies that $\eta$ extends continuously
at $z_i$ if $\eta \in \Omega^{(0,1)}_{\vec n;\vec z,p}(u^*TM)$.
Hence we have proved $\eta \equiv 0$ since we already know that $\eta = 0$
on $\Sigma \setminus\{z_0, \cdots, z_k\}$.

\section{Stratawise transversality and finiteness of ramification profiles}
\label{sec:finiteness}

In this section, we apply the stratawise transversality result to prove
a finiteness result on the types of singularities of $J$-holomorphic
maps $u : \Sigma \to M$ with fixed homology class
$f_*[\Sigma] = \beta \in H_2(M,\Z)$. The case $n=1$ corresponds to
the case where both domain and target are Riemann surfaces.
In this case, finiteness of ramification profiles follows from
the classical Hurewitz formula. Therefore we will assume $n \geq 2$ in this section.

\begin{defn}
Let $u \in \CM_{g,k}(M,J;\beta, \vec n)$.
We call the pair $(k; \vec n)$ with $k \in \N$ and $\vec n \in \N^k$
the \emph{ramification profile} of $J$-holomorphic map $u$.
\end{defn}

We have two kinds of immediate predecessors $(k';\vec n')$ to $(k;\vec n)$.
\begin{enumerate}
\item[(a)] $(k; \vec n') = (k; \vec n + \vec e_\ell)$ for some $1 \leq \ell \leq k$
where we denote $\vec n + \vec e_\ell$ the decoration
$$
(n_1, \cdots, n_\ell + 1, \cdots, n_k).
$$
\item[(b)] $(k'; \vec n') = (k+1, \vec n \cup \{n_{k+1}\})$
with $n_{k+1} = 1$.
\end{enumerate}

We have already proved that if $J \in \CJ_\omega^{ram}$, each
moduli space $\CM_{g,k}(M,J;\beta, \vec n)$ is a smooth manifold itself.

The following is an immediate consequence of Theorem \ref{thm:main}, which relates two moduli
spaces right next to each other in the partial order $<$.

\begin{thm}\label{thm:stratawise} For $J \in \JJ_\omega^{ram}$ and $\beta \in H_2(M)$ and $g \in \N$,
the following holds :
\begin{enumerate}
\item For the type $(a)$ of the immediate predecessor of $(k';\vec n')
= (k,\vec n + \vec e_\ell)$ for some $\ell = 1, \cdots, k$, $\widetilde
\MM_{g,k}(J;\beta, \vec n + \vec e_\ell)$ is a smooth submanifold of
$\widetilde \MM_{g,k}(J;\beta, \vec n)$ with its dimension $2n$ smaller,
\item For the type $(b)$, the forgetful map
$\MM_{g,k+1}(J;\beta,\vec n + \vec e_{k+1}) \to \MM_{g,k}(J;\beta,\vec n)$
is an embedding of codimension $2(n-1)$.
\end{enumerate}
\end{thm}
\begin{proof} We start with the case (1).
By definition, we have
$$
\widetilde \CM_{g,k}(M,J;\beta;\vec n + \vec e_\ell) \subset
\widetilde \CM_{g,k}(M,J;\beta;\vec n).
$$
Since we have chosen $J \in \CJ^{ram}_\omega$, Theorem 4.2 implies that
both $\widetilde \CM_{g,k}(M,J;\beta;\vec n)$ and
$\widetilde \CM_{g,k}(M,J;\beta;\vec n + \vec e_\ell)$ are smooth manifolds
and the latter has codimension $2n$ in the former.

The case of immediate predecessor of the type (2) essentially follows
from the proof of $1$-jet evaluation transversality
result of \cite{oh-zhu} (see section 2 \cite{oh-zhu} more specifically).
This finishes the proof.
\end{proof}

%To prove the submanifold property, we note that
%$$
%\widetilde \CM_{g,k}(M,J;\beta;\vec n + \vec e_\ell)
%= \left(\sigma^{\vec n + \vec e_\ell}_\ell|_{\widetilde \CM_{g,k}(M,J;\beta;\vec n)}\right)^{-1}(0)
%$$
%where
%$$
%\sigma^{\vec n + \vec e_\ell}_\ell|_{\widetilde \CM_{g,k}(M,J;\beta;\vec n)}
%: \widetilde \CM_{g,k}(M,J;\beta;\vec n) \to H^{(n_\ell+1,0)}|_{\widetilde \CM_{g,k}(M,J;\beta;\vec n)}
%$$
%is the restriction of the section
%\be\label{eq:sectionsigma}
%\sigma^{\vec n + \vec e_\ell}_{\ell;J}: \CF_k(M,\beta;\vec n)
%\to H^{(n_\ell+1,0)}_J.
%\ee
%Since we have chosen $J \in \CJ^{ram}_\omega$, $\widetilde \CM_{g,k}(M,J;\beta;\vec n)$
%is a smooth manifold and
%$$
%H^{(n_\ell+1,0)}|_{\widetilde \CM_{g,k}(M,J;\beta;\vec n)}
%\to \widetilde \CM_{g,k}(M,J;\beta;\vec n)
%$$
%defines a smooth vector bundle of rank $2n$. The linearization of this section
%is nothing but the restriction of the linearization of \eqref{eq:sectionsigma}.
%
%Since $J \in \CJ_\omega^{ram}$, the codimension of $
%
%the section
%$\sigma^{\vec n + \vec e_\ell}_\ell|_{\widetilde \CM_{g,k}(M,J;\beta;\vec n)}$
%is transverse to the zero section of
%$H^{(n_\ell+1,0)}|_{\widetilde \CM_{g,k}(M,J;\beta;\vec n)}$ which
%proves the theorem.
%\end{proof}

Another immediate consequence of Theorem \ref{thm:main} is the following finiteness result.

\begin{thm}\label{thm:finite} Let $\beta \in H_2(M,\Z)$ and $g$ be given.
Then for any $J \in \CJ_\omega^{ram}$,  the number of
types of ramification profiles is not bigger than
$$
P(c_1(\b)+(3-n)(g-1)),
$$
that is, the number of partitions of the integer
$c_1(\b)+(3-n)(g-1)$.
\end{thm}
\begin{proof}
We have the natural projection
$$
\pi:\widetilde \CM_{g,k}(M,\beta;\vec n) : = \bigcup_{J \in \CJ_\omega}
\widetilde \CM_{g,k}(M,J;\beta;\vec n) \to \CJ_\omega.
$$
The projection has index
$$
2(c_1(\beta) + n(1-g)) + 2k - \sum_{\ell}^k 2n n_\ell =
2(c_1(\beta) + n(1-g)) - \sum_{\ell}^k 2 (n n_\ell-1)
$$ so for any regular value $J$, the moduli space
$$
\widetilde \CM_{g,k}(M,J;\beta;\vec n) =\Upsilon_\ell^{-1} (o_{\CH''} \times
o_{H^{(n_\ell,0)}} )\cap \pi^{-1}(J)
$$
is of dimension
$$
2(c_1(\beta) + n(1-g)) - \sum_{\ell}^k 2 (n n_\ell-1).
$$

We consider the quotient
$$
\CM_{g,k}(M,J;\beta;\vec n)
:=\widetilde \CM_{g,k}(M,J;\beta;\vec n)/\Aut(\Sigma),
$$
where $Aut(\Sigma)$ acts on marked Riemann surfaces $(\S,j,z)$ by
conformal equivalence then on the maps from them.
For any $J \in \CJ_\omega^{ram}$, $\CM_{g,k}(M,J;\beta;\vec n)$ has its dimension
$$
2 \left(c_1(\b)+(3-n)(g-1) - \sum_{\ell}^k (n n_\ell-1)\right)
$$
as a smooth orbifold. Therefore, $\CM_{g,k}(M,J;\beta;\vec n)$ is
empty whenever this dimension is negative. In other words, if
$\CM_{g,k}(M,J;\beta;\vec n) \neq \emptyset$, then we should have
\be\label{eq:nell} \sum_{\ell}^k (n n_\ell-1) \leq
c_1(\b)+(3-n)(g-1). \ee In particular if $c_1(\b)+(3-n)(g-1) \leq
0$, then for any element $(j,u) \in \CM_{g,k}(M,J;\beta)$, $u$ will
be immersed.

On the other hand, We note that since we assume $n \geq 2$, $nn_\ell
-1 \geq n-1 > 0$. Therefore if $c_1(\b)+(3-n)(g-1) > 0$, the
inequality \eqref{eq:nell} implies that the number of admissible
pairs $(k;\vec n)$ is not bigger than
$$
P(c_1(\b)+(3-n)(g-1)),
$$
that is, the number of partitions of the integer $c_1(\b)+(3-n)(g-1)
$. This then finishes the proof of Theorem \ref{thm:finite}.
\end{proof}

\end{document}